\documentclass[12pt]{amsart}
\usepackage{mathtools}
\usepackage{amsmath, amsthm, amssymb}
\usepackage{listings}
\usepackage{url}
\usepackage{hyperref}
\usepackage[utf8]{inputenc}
\usepackage{graphicx}

\begin{document}

\theoremstyle{definition}
\newtheorem{mydef}{Definition}
\newtheorem{lem}[mydef]{Lemma}
\newtheorem{thm}[mydef]{Theorem}
\newtheorem{prop}[mydef]{Proposition}
\newtheorem{fact}[mydef]{Fact}
\newtheorem{example}[mydef]{Example}
\newtheorem{remark}[mydef]{Remark}
\newtheorem{cor}[mydef]{Corollary}
\newtheorem{claim}{Claim}
\newtheorem{question}[mydef]{Question}

\newcommand{\gn}[1]{\lceil #1 \rceil}

\newcommand{\fct}[2]{\prescript{#1}{}{#2}}

\newcommand{\concat}{%
  \mathord{
    \mathchoice
    {\raisebox{1ex}{\scalebox{.7}{$\frown$}}}
    {\raisebox{1ex}{\scalebox{.7}{$\frown$}}}
    {\raisebox{.5ex}{\scalebox{.5}{$\frown$}}}
    {\raisebox{.5ex}{\scalebox{.5}{$\frown$}}}
  }
}

\def\forkindep{\mathrel{\raise0.2ex\hbox{\ooalign{\hidewidth$\vert$\hidewidth\cr\raise-0.9ex\hbox{$\smile$}}}}}

\newcommand{\indep}[3]{#2 \underset{#1}{\forkindep} #3}

\parindent 0pt
\parskip 5pt

\title{Indiscernible extraction and Morley sequences}

\author{Sebastien Vasey}
\email{sebv@cmu.edu}
\address{Department of Mathematical Sciences, Carnegie Mellon University, Pittsburgh, Pennsylvania, USA}

\begin{abstract}
  We present a new proof of the existence of Morley sequences in simple theories. We avoid using the Erdős-Rado theorem and instead use only Ramsey's theorem and compactness. The proof shows that the basic theory of forking in simple theories can be developed using only principles from ``ordinary mathematics'', answering a question of Grossberg, Iovino and Lessmann, as well as a question of Baldwin.
\end{abstract}

\date{\today\\
AMS 2010 Subject Classification: Primary:  03C45. Secondary: 03B30, 03E30.}

\keywords{Forking; Morley sequences; Dual finite character; Simple theories}

\maketitle

\section{Introduction}

Shelah \cite[Lemma 9.3]{sh93} has shown that, in a simple first-order theory $T$, Morley sequences exist for every type. The proof proceeds by first building an independent sequence of length $\beth_{\left(2^{|T|}\right)^+}$ for the given type and then using the Erdős-Rado theorem together with Morley's method to extract the desired indiscernibles. 

After slightly improving on the length of the original independent sequence \cite[Appendix A]{primer}, Grossberg, Iovino and Lessmann observed that, in contrast, most of the theory of forking in a stable first-order theory $T$ does not need the existence of such ``big'' cardinals. The authors then asked whether the same could be said about simple theories, and so in particular whether there was another way to build Morley sequences there.

Baldwin (see \cite{baldwinfom} and \cite[Question 3.1.9]{baldwin-monster}) similarly asked\footnote{Akito Tsuboi \cite{tsuboi-dividing} has independently answered this question.} whether the equivalence between forking and dividing in simple theories had an alternative proof.

We answer those questions in the affirmative by showing how to extract a Morley sequence from any infinite independent sequence. Our construction relies on a property of forking we call \emph{dual finite character}. We show it holds in simple theories, and that the converse is also true (the latter was noticed by Itay Kaplan).

This paper was written while working on a Ph.D.\ thesis under the direction of Rami
Grossberg at Carnegie Mellon University and I would like to thank Professor Grossberg for his guidance and assistance in my research in general and in this work
specifically. I also thank John Baldwin, José Iovino, Itay Kaplan, Alexei Kolesnikov, Anand Pillay, and Akito Tsuboi for valuable comments on earlier versions of this paper.

\section{Preliminaries}

For the rest of this paper, fix a complete first-order theory $T$ in a language $L(T)$ and work inside its monster model $\mathfrak{C}$. We write $|T|$ for $|L(T)| + \aleph_0$. We denote by $\text{Fml} (L(T))$ the set of first-order formulas in the language $L(T)$. If $A$ is a set, we say a formula is \emph{over $A$} if it has parameters from $A$. For a tuple $\bar{a}$ in $\mathfrak{C}$ and $\phi$ a formula, we write $\models \phi[\bar{a}]$ instead of $\mathfrak{C} \models \phi[\bar{a}]$.

When $I$ is a linearly ordered set, $(\bar{a}_i)_{i \in I}$ are tuples, and $i \in I$, we write $\bar{a}_{<i}$ for $(\bar{a}_j)_{j < i}$. It is often assumed without comments that all the $\bar{a}_i$s have the same (finite) arity.

We assume the reader is familiar with forking. We will use the combinatorial definition stated e.g.\ in \cite[Definition 1.2]{sh93}. It turns out that our construction of Morley sequences does not rely on this exact definition, but only on abstract properties of forking such as invariance, extension, and symmetry.

Recall also the definition of a Morley sequence:

\begin{mydef}
  Let $I$ be a linearly ordered set. Let $\mathbf{I} := \langle \bar{a}_i \mid i \in I\rangle$ be a sequence of finite tuples of the same arity. Let $A \subseteq B$ be sets, and let $p \in S (B)$ be a type that does not fork over $A$.

  $\mathbf{I}$ is said to be an \emph{independent sequence for $p$ over $A$} if:
  \begin{enumerate}
    \item For all $i \in I$, $\bar{a}_i \models p$.
    \item For all $i \in I$, $\text{tp} (\bar{a}_i / B \bar{a}_{<i})$ does not fork over $A$. 
  \end{enumerate}

  $\mathbf{I}$ is said to be a \emph{Morley sequence for $p$ over $A$} if:
  
  \begin{enumerate}
    \item $\mathbf{I}$ is an independent sequence for $p$ over $A$.
    \item $\mathbf{I}$ is indiscernible over $B$.
  \end{enumerate}
\end{mydef}



\section{Morley sequences in simple theories}

It is well known that independent sequences can be built by repeated use of the extension property of forking. If the theory is stable, the existence of Morley sequences follows, because in such theories any sufficiently long sequence contains indiscernibles. The latter fact is no longer true in general, and in fact there are counterexamples among both simple \cite[p.~209]{sh197} and dependent \cite{kpsh975} theories. Thus a different approach is needed in the unstable case. Recall from the introduction that we do not want to use big cardinals, so Morley's method cannot be used. We can however make use of the following variation of the Ehrenfeucht-Mostowski theorem:

\begin{fact}[\cite{tentzieglerbook}, Lemma 5.1.3]\label{weak-indisc-extraction}
  Let $A$ be a set, and let $I$ be a linearly ordered set. Let $\mathbf{J} := \left< \bar{a}_j \mid j < \omega\right>$ be a sequence of finite tuples of the same arity. Then there exists a sequence $\mathbf{I} := \left< \bar{b}_i \mid i \in I \right>$, indiscernible over $A$ such that:

  For any $i_0 < \ldots < i_{n - 1}$ in $I$, for all \emph{finite} $q \subseteq \text{tp} (\bar{b}_{i_0} \dots \bar{b}_{i_{n - 1}} / A)$, there exists $j_0 < \ldots < j_{n - 1} < \omega$ so that $\bar{a}_{j_0} \dots \bar{a}_{j_{n - 1}} \models q$.
\end{fact}

Do we get a Morley sequence if we apply Fact \ref{weak-indisc-extraction} to an independent sequence? In general, we see no reason why it should be true. However, we will see that it \emph{is} true if we assume the following local definability property of forking:

\begin{mydef}[Dual finite character]\label{def-dfc}
  Forking is said to have \emph{dual finite character (DFC)} if whenever $\text{tp} (\bar{c} / A \bar{b})$ forks over $A$, there is a formula $\phi (\bar{x}, \bar{y})$ over $A$ such that:

  \begin{itemize}
    \item $\models \phi[\bar{c}, \bar{b}]$, and:
    \item $\models \phi[\bar{c}, \bar{b}']$ implies $\text{tp} (\bar{c} / A \bar{b}')$ forks over $A$.
  \end{itemize}
\end{mydef}

A variation of DFC appears as property A.7' in \cite{makkai-survey}, but we haven't found any other explicit occurrence in the literature. Notice that DFC immediately implies something stronger:

\begin{prop}\label{lem-dfc}
  Assume forking has DFC. Assume $p := \text{tp} (\bar{c} / A \bar{b})$ forks over $A$, and $\phi (\bar{x}, \bar{y})$ is as given by Definition \ref{def-dfc}. Then $\text{tp} (\bar{c}' / A) = \text{tp} (\bar{c} / A)$ and $\models \phi[\bar{c}', \bar{b}']$ imply $\text{tp} (\bar{c}' / A \bar{b}')$ forks over $A$.
\end{prop}
\begin{proof}
  Assume $\text{tp} (\bar{c}' / A) = \text{tp} (\bar{c} / A)$. Let $f$ be an automorphism of $\mathfrak{C}$ fixing $A$ such that $f (\bar{c}') = \bar{c}$. Assume $\models \phi[\bar{c}', \bar{b}']$. Applying $f$, $\models \phi[\bar{c}, f(\bar{b}')]$. Since $\phi$ witnesses DFC, $\text{tp} (\bar{c} / A f (\bar{b}'))$ forks over $A$. Applying $f^{-1}$ and using invariance of forking, $\text{tp} (\bar{c}' / A \bar{b}')$ forks over $A$.
\end{proof}


\begin{thm}\label{dfc-morley}
  Assume forking has DFC\@. Let $A \subseteq B$ be sets. Let $p \in S (B)$ be a type that does not fork over $A$. Let $I$ be a linearly ordered set. Then there is a Morley sequence $\mathbf{I} := \left<\bar{b}_i \mid i \in I\right>$ for $p$ over $A$.
\end{thm}
\begin{proof}
  By repeated use of the extension property of forking, build an independent sequence $\mathbf{J} := \left< \bar{a}_j \mid j < \omega \right>$ for $p$ over $A$. 

  Let $\mathbf{I} := \left< \bar{b}_i \mid i \in I\right>$ be indiscernible over $B$ as described by Fact \ref{weak-indisc-extraction}. We claim $\mathbf{I}$ is as required. 

  It is indiscernible over $B$, and for every $i \in I$, every $\bar{b}_i$ realizes $p$: If $\bar{b}_i \not \models p$, fix a formula $\phi (\bar{x}, \bar{b}) \in p$ so that $\models \neg \phi[\bar{b}_i, \bar{b}]$. By the defining property of $\mathbf{I}$, there exists $j < \omega$ so that $ \models \neg \phi[\bar{a}_j, \bar{b}]$, so $\bar{a}_j \not \models p$, a contradiction.

It remains to see that for every $i \in I$, $p_i := \text{tp} (\bar{b}_i / B \bar{b}_{<i})$ does not fork over $A$. Assume not, and fix $i \in I$ so that $p_i$ forks over $A$. Fix $\bar{b} \in B$ and $i_0 < \ldots < i_{n - 1} < i$ such that $p_i' := \text{tp} (\bar{b}_i / A  \bar{b}_{i_0} \dots \bar{b}_{i_{n - 1}} \bar{b})$ forks over $A$. Fix $\phi (\bar{x}, \bar{b}_{i_0} \dots \bar{b}_{i_{n - 1}} \bar{b}) \in p_i'$ a formula over $A$ witnessing DFC\@. 

Find $j_0 < \ldots < j_n < \omega$ such that $\models \phi[\bar{a}_{j_n}, \bar{a}_{j_0} \dots \bar{a}_{j_{n - 1}}\bar{b}]$. Since it has already been observed that $\text{tp} (\bar{a}_{j_n} / A) = \text{tp} (\bar{b}_i / A) = p \upharpoonright A$, Proposition \ref{lem-dfc} implies that $\text{tp} (\bar{a}_{j_n} / A \bar{a}_{j_0} \dots \bar{a}_{j_{n - 1}} \bar{b})$ forks over $A$, contradicting the independence of $\mathbf{J}$.
\end{proof}

We now show that a simple theory has DFC (this was essentially already observed by Makkai). Recall \cite[Theorem 2.4]{kim-sym} that $T$ is simple exactly when forking has the symmetry property. Moreover, the methods of \cite{adler-rank} show that the equivalence can be proven without using Morley sequences. The key is \cite[Theorem 3.6]{adler-rank}, which shows (without using Morley sequences) that if the $D$-rank is bounded, then symmetry holds.

\begin{lem}\label{simple-dfc}
  Assume $T$ is simple. Then forking has DFC.
\end{lem}
\begin{proof}
  Assume $p := \text{tp} (\bar{c} / A \bar{b})$ fork over $A$. By symmetry, $q := \text{tp} (\bar{b} / A \bar{c})$ forks over $A$. Fix $\psi (\bar{y}, \bar{x})$ over $A$ such that $\psi (\bar{y}, \bar{c}) \in q$ witnesses forking, i.e.\ if $\models \psi[\bar{b}', \bar{c}]$ then $\text{tp} (\bar{b}' / A \bar{c})$ forks over $A$.

    Let $\phi (\bar{x}, \bar{y}) := \psi (\bar{y}, \bar{x})$. Then $\phi (\bar{x}, \bar{b}) \in p$, and if $\models \phi[\bar{c}, \bar{b}']$, then $\models \psi[\bar{b}', \bar{c}]$, so $\text{tp} (\bar{b}' / A \bar{c})$ forks over $A$, so by symmetry, $\text{tp} (\bar{c} / A \bar{b}')$ forks over $A$. This shows $\phi (\bar{x}, \bar{y})$ witnesses DFC.
\end{proof}

\begin{cor}[Existence of Morley sequences in simple theories]\label{simple-morley}
  Assume $T$ is simple. Let $A \subseteq B$ be sets. Let $p \in S (B)$ be a type that does not fork over $A$. Let $I$ be a linearly ordered set. Then there is a Morley sequence $\mathbf{I} := \left<\bar{b}_i \mid i \in I\right>$ for $p$ over $A$.
\end{cor}
\begin{proof}
  Combine Lemma \ref{simple-dfc} and Theorem \ref{dfc-morley}.
\end{proof}

We end by closing the loop on our study of DFC: Lemma \ref{simple-dfc} shows that simplicity implies DFC, but it turns out that they are equivalent! This was pointed out by Itay Kaplan in a personal communication. Definition \ref{weak-dfc-def} and (\ref{eq-dfc}) implies (\ref{eq-weak-dfc}) implies (\ref{eq-tsimple}) in Theorem \ref{dfc-simple} below are due to Kaplan, and I am grateful to him for allowing me to include them here. 

The key is to observe that symmetry fails very badly when the theory is not simple:

\begin{fact}[\cite{chernikov-ntp2}, Lemma 6.16]\label{fail-symmetry}
  Assume $T$ is \emph{not} simple. Then there is a model $M$ and tuples $\bar{b}, \bar{c}$ such that $\text{tp} (\bar{b} / M \bar{c})$ is finitely satisfiable in $M$, but $\text{tp} (\bar{c} / M \bar{b})$ divides over $M$.
\end{fact}

We are now ready to prove that forking has DFC exactly when the theory is simple. In fact, we only need the following version of DFC:

\begin{mydef}\label{weak-dfc-def}
  Forking is said to have \emph{weak dual finite character (weak DFC)} if whenever $M$ is a model and $\text{tp} (\bar{c} / M \bar{b})$ divides over $M$, there is a formula $\phi (\bar{x}, \bar{y})$ over $M$ such that:

  \begin{itemize}
    \item $\models \phi[\bar{c}, \bar{b}]$, and:
    \item $\models \phi[\bar{c}, \bar{b}']$ implies $\text{tp} (\bar{c} / M \bar{b}')$ is not finitely satisfiable in $M$.
  \end{itemize}
\end{mydef}

\begin{thm}\label{dfc-simple}
  The following are equivalent:

  \begin{enumerate}
    \item\label{eq-tsimple} $T$ is simple.
    \item\label{eq-dfc} Forking has DFC.
    \item\label{eq-weak-dfc} Forking has weak DFC.
  \end{enumerate}
\end{thm}
\begin{proof}
  (\ref{eq-tsimple}) implies (\ref{eq-dfc}) is Lemma \ref{simple-dfc}, and (\ref{eq-dfc}) implies (\ref{eq-weak-dfc}) is because finite satisfiability implies nonforking. We show (\ref{eq-weak-dfc}) implies (\ref{eq-tsimple}). Assume $T$ is not simple. Fix $M$ and $\bar{b}$, $\bar{c}$ as given by Fact \ref{fail-symmetry}. In particular, $p := \text{tp} (\bar{c} / M \bar{b})$ divides over $M$. Let $\phi (\bar{x}, \bar{y})$ be a formula over $M$ such that $\models \phi[\bar{c}, \bar{b}]$. By assumption, $\text{tp} (\bar{b} / M\bar{c})$ is finitely satisfiable in $M$, so in particular there is $\bar{b}' \in M$ such that $\models \phi[\bar{c}, \bar{b}']$. Thus $\text{tp} (\bar{c} / M \bar{b}') = \text{tp} (\bar{c} / M)$ must be finitely satisfiable over $M$, hence $\phi (\bar{x}, \bar{y})$ cannot witness weak DFC for $p$. Since $\phi$ was arbitrary, this shows weak DFC fails.
\end{proof}

We end by pointing out that all the results of this paper could be formalized in a weak fragment of ZFC, such as ZFC - Replacement - Power set + ``For any set $X$ of size $\le |T|$, $\mathcal{P} (\mathcal{P} (X))$ exists''\footnote{Formally, we have to work in work in a language containing a constant symbol standing for $|T|$.}. Going further, it would be interesting to extend Harnik's work on the reverse mathematics of stability theory \cite{harnik85, harnik87} by finding the exact proof-theoretic strength of the existence of Morley sequences.

 \bibliographystyle{amsalpha}
 \bibliography{morley-seq}

\end{document}